\documentclass{article}
\usepackage{amsfonts}
\usepackage{amsmath,amssymb,amsthm,color}
\usepackage[utf8]{inputenc}
\usepackage{amsopn}
\usepackage{latexsym}
\usepackage{float}
\usepackage{url}
\usepackage{multicol}
\usepackage{hyperref}
\usepackage{multicol}
\usepackage{rotating}
\usepackage{multirow}
\usepackage{tikz-cd}
\usepackage{tikz}

\def \squarebox#1{\hbox to #1{\hfill \vbox to #1{\vfill}}}

\newtheorem{teo}{Theorem}

\newtheorem{pps}[teo]{Proposition}
\newtheorem{cor}[teo]{Corollary}
\newtheorem{lem}[teo]{Lemma}

\newtheorem{example}[teo]{Example}
\newtheorem*{remark}{Remark}

\DeclareMathOperator{\Ad}{Ad}
\DeclareMathOperator{\ad}{ad}

\DeclareMathOperator{\tr}{tr}
\DeclareMathOperator{\id}{id}

\begin{document}

\title{\bf Deformations of Adjoint orbits for semisimple Lie algebras and Lagrangian submanifolds }
\author{Jhoan Báez and Luiz A. B. San Martin\thanks{JB was supported by CNPq grant nº 141173/2019-0 and LSM by  CNPq grant nº 305513/2003-6 and FAPESP grant nº 07/06896-5, Address: Imecc - Unicamp, Departamento de Matemática. Rua Sérgio Buarque de Holanda, 651, Cidade Universitária Zeferino Vaz. 13083-859 Campinas - SP, Brasil. E-mails: smartin@ime.unicamp.br, sebastianbaezzz@gmail.com. }}

\maketitle

\date{}

\begin{abstract}
We give a coadjoint orbit's diffeomorphic deformation between the classical semisimple case and the semi-direct product given by a Cartan decomposition. The two structures admit the Hermitian symplectic form defined in a semisimple complex Lie algebra. We provide some applications such as the constructions of Lagrangian submanifolds. 
\end{abstract}

\textit{AMS 2010 subject classification:} 14M15, 22F30, 53D12. 

\textit{Key
words and phrases:} Adjoint orbits, Homogeneous space, Hermitian
symplectic form, Lagrangian submanifolds.


\section{Introduction}

Let $\mathfrak{g}$ be a non-compact semisimple Lie group with Cartan decomposition $\mathfrak{g}=\mathfrak{k}\oplus \mathfrak{s}$ and Iwasawa decomposition $\mathfrak{g}=\mathfrak{k}\oplus \mathfrak{a}\oplus \mathfrak{n}$ with $\mathfrak{a}\subset \mathfrak{s}$ maximal abelian. In the vector space underlying $\mathfrak{g}$ there is another Lie algebra structure $\mathfrak{k}_{\ad}=\mathfrak{k}\times _{\ad} \mathfrak{s}$ given by the semi-direct product defined by the adjoint representation of $\mathfrak{k}$ in $\mathfrak{s}$, which is viewed as an abelian Lie algebra.

Let $G=\mathrm{Aut}_{0}\mathfrak{g}$ be the adjoint group of $\mathfrak{g}$ (identity component of the automorphism group) and put $K=\exp \mathfrak{k}\subset G$. The semi-direct product $K_{\ad}=K\times_{\Ad} \mathfrak{s}$ obtained by the adjoint representation of $K$ in $\mathfrak{s}$ has Lie algebra $\mathfrak{k}_{\mathrm{ad}}=\mathfrak{k}\times _{\ad} \mathfrak{s}$ (see Subsection \ref{sdpcd}).

In this paper we consider coadjoint orbits for both Lie algebras $\mathfrak{g}$ and $\mathfrak{k}_{\ad}$. These orbits are submanifolds of $\mathfrak{g}^{\ast }$ that we identify with $\mathfrak{g}$ via the Cartan-Killing form of $\mathfrak{g}$, so that the orbits are seen as submanifolds of $\mathfrak{g}$. These are just the adjoint orbits for the Lie algebra $\mathfrak{g}$ while for $\mathfrak{k}_{\ad}$ they are the orbits in $\mathfrak{g}$ of the representation of $K_{\ad}$ obtained by transposing its coadjoint representation. The orbits through $H\in \mathfrak{g}$ are denoted by $\Ad \left( G\right) \cdot H$ and $K_{\ad}\cdot H$, respectively.

We consider the orbits through $H\in \mathfrak{a}\subset \mathfrak{s}$. In this case the compact orbit $\Ad \left( K\right) \cdot H$ (contained in $\mathfrak{s}$) is a flag manifold of $\mathfrak{g}$, say $\mathbb{F}_{H}$. In Gasparim-Grama-San Martin \cite{smgasgr1} it was proved that $\Ad \left( G\right) \cdot H$ is diffeomorphic to the cotangent space $T^{\ast } \mathbb{F}_{H}$ of $\mathbb{F}_{H}=\Ad \left( K\right) \cdot H$. We prove here that the same happens to the semi-direct product orbit $K_{\ad}\cdot H$ (as foreseen by Jurdjevic \cite{jur}). So that $\Ad \left( G\right) \cdot H$ and $K_{\ad}\cdot H$ diffeomorphic to each other.

In this paper we define a deformation $\mathfrak{g}_{r}$ of the original Lie algebra $\mathfrak{g}$ (see Section \ref{deforb}). The deformation is parameterized by $r>0$ and satisfies $\mathfrak{g}_{1}=\mathfrak{g}$. For each $r$ the Lie algebra $\mathfrak{g}_{r}$ is isomorphic to $\mathfrak{g}$ (hence semisimple) and $\mathfrak{g}_{r}=\mathfrak{k}\oplus \mathfrak{s}$ is a Cartan decomposition as well with $\mathfrak{k}$ a subalgebra of $\mathfrak{g}_{r}$. Furthermore as $r\rightarrow \infty $ the Lie algebra $\mathfrak{k}_{\ad }$ is recovered. (The deformation amounts essentially to change the brackets $\left[ X,Y\right] $, $X,Y\in \mathfrak{s}$, by $\left( 1/r\right) \left[ X,Y \right] $ and keeping the other brackets unchanged.) A Lie algebra $\mathfrak{g}_{r}$, $r>0$, has its own automorphism group whose identity component is denoted by $G_{r}$. Thus the adjoint orbits in $ \mathfrak{g}_{r}$ are $\Ad (G_{r})\cdot H$ and by the isomorphism $  \mathfrak{g}_{r}\approx \mathfrak{g}$ it follows that $\Ad (G_{r})\cdot H$ is diffeomorphic to $\Ad (G)\cdot H$ and hence to the cotangent space $T^{\ast }\mathbb{F}_{H}$. Thus the Lie algebra deformation yields a continuous one parameter family of embeddings of $T^{\ast }\mathbb{F}_{H}$ into the vector space underlying $\mathfrak{g}$. The family is parameterized in $(0,+\infty ]$ where $+\infty $ is the embedding given by the the semi-direct product orbit $K_{\ad}\cdot H$.

The example with $\mathfrak{g}=\mathfrak{sl}\left( 2,\mathbb{R}\right) $, presented in Subsection \ref{sdpcd}, is elucidative of this deformation. In $\mathfrak{sl}\left( 2,\mathbb{R}\right) \approx \mathbb{R}^{3}$ the semi-direct product orbit is the cylinder $x^{2}+y^{2}=1$ while the adjoint orbit is the one-sheet hyperboloid $x^{2}+y^{2}-z^{2}=1$. In the deformation the adjoint orbit in $\mathfrak{g}_{r}$ is the paraboloid $x^{2}+y^{2}-z^{2}/r=1$ that converges to the cylinder as $r\rightarrow +\infty $. The hyperboloids as well as the cylinder are unions of straight lines in $\mathbb{R}^{3}$ crossing the circle $x^{2}+y^{2}=1$ with $z=0$. As is well known the hyperboloids are obtained by twisting the generatrices of the cylinder.

This picture of \textquotedblleft twisting generatrices\textquotedblright\ holds in a general Lie algebra $\mathfrak{g}$: The semi-direct product orbit  $K_{\ad}\cdot H$ has the cylindrical shape 
\begin{equation*}
K_{\ad }\cdot H=\bigcup\limits_{X\in \Ad \left( K\right)
H}\left( X+\ad \left( X\right) \mathfrak{s}\right)
\end{equation*}
where $\ad \left( X\right) \mathfrak{s}$ is a subspace of $\mathfrak{k}$. While the adjoint orbit $\Ad (G)\cdot H$ has the hyperboloid shape
\begin{equation*}
\Ad (G)\cdot H=\bigcup\limits_{k\in K} \Ad \left( k\right)
\left( H+\mathfrak{n}_{H}^{+}\right)
\end{equation*}
where $\mathfrak{n}_{H}^{+}$ is the nilpotent subalgebra which is the sum of the eigenspaces of $\ad \left( H\right) $ associated to positive eigenvalues. The deformation of $\mathfrak{g}$ into $\mathfrak{g}_{r}$ has the effect of twisting the generatrix $H+\ad \left( H\right)  \mathfrak{s}\subset \mathfrak{k}$ into $H+\mathfrak{n}_{r,H}^{+}$ where $ \mathfrak{n}_{r,H}^{+}$ the nilpotent Lie subalgebra of $\mathfrak{g}_{r}$ defined the same way as $\mathfrak{n}^{+}$ from the adjoint $\ad_{r}\left( H\right) $ of $H$ in $\mathfrak{g}_{r}$.
The deformation of orbits allows to transfer geometric properties from semi-direct product orbit $K_{\ad }\cdot H$ to the adjoint orbit $ \Ad \left( G\right) \cdot H$. This can be useful since in several aspects the geometry of $K_{\ad }\cdot H$ is more manageable than that of $\Ad \left( G\right) \cdot H$.

In this paper we apply this transfer approach to adjoint orbits in a complex semisimple Lie algebra $\mathfrak{g}$ (see Section \ref{csalg}). In the complex case $\mathfrak{g}$ is endowed with a Hermitian metric
\begin{equation*}
\mathcal{H}_{\tau }\left( X,Y\right) =\langle X,Y\rangle +i\Omega \left(
X,Y\right)
\end{equation*}%
where $\langle \cdot ,\cdot \rangle $ is an inner product and $\Omega $ is a symplectic form. The form $\Omega $ restricts to symplectic forms on the orbit $\Ad \left( G\right) \cdot H$ since this is a complex submanifold. Although not as immediate we prove that the restriction of $\Omega $ to the semi-direct product orbit $K_{\ad }\cdot H$ is also a symplectic form. By construction the diffeomorphisms between the adjoint orbits (including $K_{\ad}\cdot H$) are symplectomorphisms.
Based on these facts in Section \ref{secLagran} we construct Lagrangian submanifolds in $K_{\ad }\cdot H$ (w.r.t. $\Omega $) and then transport them to $\Ad_r \left( G\right) \cdot H$ through the deformation. To conclude in section \ref{lsadact}, we describe the Lagrangian orbits given by the adjoint action of $U$ and the Hermitian form $\Omega$ on $\Ad \left( G\right) \cdot H$.


\section{Semi-direct products} \label{semidirprod}

Let $G$ be a compact connected Lie group  with Lie algebra $\mathfrak{g}$ and take a representation $\rho :G\rightarrow \mathrm{Gl}\left( V\right) $ on a vector space $V$ (with $\dim V<\infty$). The infinitesimal representation of $\mathfrak{g}$ on $\mathfrak{gl}(V)$ is also going to be denoted by $\rho$.

In particular, the vector space $V$ can be seen as an abelian Lie group (or abelian Lie algebra). In this way, we can take the semi-direct product $G\times _{\rho}V $ which is a Lie group whose underlying manifold is the Cartesian product  $G\times V$. This group is going to be denoted by $G_{\rho }$ and its Lie algebra $\mathfrak{g}_{\rho }$ is the semi-direct product 
\begin{equation*}
\mathfrak{g}_{\rho }=\mathfrak{g}\times _{\rho }V.
\end{equation*}

The vector space of $\mathfrak{g}_{\rho }$ is $\mathfrak{g}\times V$ with bracket 
\begin{equation*}
\left[ \left( X,v\right) ,\left( Y,w\right) \right] =\left( \left[ X,Y\right] ,\rho \left( X\right) w-\rho \left( Y\right) v\right) .
\end{equation*}

Our purpose is to describe the coadjoint orbit on the dual $\mathfrak{g}_{\rho }^{\ast }$ of $\mathfrak{g}_{\rho }$. To begin, let's see how to determine the $\rho$-adjoint representation  $\ad_{\rho }\left( X,v\right) $, with $\left( X,v\right) \in \mathfrak{g}\times _{\rho}V$. Thus, take a basis of $\mathfrak{g}\times V$ denoted by $\mathcal{B} = \mathcal{B}_{\mathfrak{g}}\cup \mathcal{B}_{V}$ with $\mathcal{B}_{\mathfrak{g}}=\{X_{1},\ldots ,X_{n}\}$ and $\mathcal{B}_{V}=\{v_{1},\ldots ,v_{d}\}$ basis of $\mathfrak{g}$  and $V$, respectively. On this basis the matrix of $\ad_{\rho}\left(X,v\right) $ is given by
\begin{equation*}
\ad_{\rho}\left( X,v\right) =\left(
\begin{array}{cc}
\ad \left( X\right) & 0 \\ 
A\left( v\right) & \rho \left( X\right)%
\end{array}%
\right) ,
\end{equation*}%
where $\ad \left( X\right) $ is the adjoint representation of $\mathfrak{g}$ while for each $v\in V$, $A\left( v\right) $ is the linear map  $\mathfrak{g}\rightarrow V$ defined by 
\begin{equation*}
A\left( v\right) \left( X\right) =\rho \left( X\right) \left( v\right) .
\end{equation*}

Since $G$ is compact, then $V$ admits a $G$-invariant inner  product $\langle \cdot ,\cdot \rangle$ (Hermitian, in the complex case). Define a map $\mu :V\otimes V\rightarrow \mathfrak{g}^{\ast}$ (called moment map of $\rho$) given by 
\begin{equation*}
    \mu \left( v \otimes w\right) \left( X\right) =\langle \rho \left( X\right) v,w\rangle ,
\end{equation*}
where $\rho \left( g\right) $  is an isometry and $\rho \left( X\right) $  is an anti-symmetric linear application of $\langle \cdot ,\cdot \rangle$ for $g\in G$ and $X\in \mathfrak{g}$, then $\mu$  is anti-symmetric. Hence the moment map $\mu $ is defined in the exterior product $\wedge ^{2}V=V\wedge V$.  Furthermore, a compact Lie algebra $\mathfrak{g}$ admits an $\ad$-invariant inner product such that we can identify $\mathfrak{g}^{\ast }$ with $\mathfrak{g}$, then $$\mu :V\wedge V\rightarrow \mathfrak{g}.$$

Similarly, the dual $\mathfrak{g}^{\ast }\times V^{\ast }$ of  $\mathfrak{g}_{\rho }=\mathfrak{g}\times V$ are identified by its inner product. This means that the coadjoint representation of $\mathfrak{g}\times V$ is written in $\mathfrak{g}\times V$ as type matrices (on orthonormal bases): 
\begin{equation*}
\ad_{\rho}^{\ast }\left( X,v\right) =\left( 
\begin{array}{cc}
\ad\left( X\right) & -A\left( v\right) \\ 
0 & \rho \left( X\right)%
\end{array}%
\right) \qquad X\in \mathfrak{g},~v\in V
\end{equation*}%
where for each $v\in V$, $A\left( v\right) :V\rightarrow \mathfrak{g}$ can be identified by  $A\left( v\right) \left( w\right) =\mu \left( v\wedge w\right)$. Then the representations $\Ad_{\rho }$ and $\Ad_{\rho }^{\ast }$ of $G_{\rho }$ are obtained by exponentials of representations in $\mathfrak{g}_{\rho }$. In particular, the following matrices are obtained (on the basis given above): 
\begin{equation}\label{forexponencial}
 e^{t\ad_{\rho }\left( 0,v\right) } =\left( 
\begin{array}{cc}
1 & 0 \\ 
tA\left( v\right) & 1%
\end{array}%
\right) ,\qquad  e^{t\ad_{\rho }^{\ast }\left(
0,v\right) } =\left( 
\begin{array}{cc}
1 & -tA\left( v\right)  \\ 
0 & 1%
\end{array}%
\right) .  
\end{equation}%
On the other hand, for $g\in G$ the restriction of $\Ad_{\rho }\left( g\right) $ to $V$ coincides with $\rho \left( g\right) $.
Thus the coadjoint orbit can be described by the moment map.
Additionally, The moment map $\mu $ is bilinear in $V \times V$, then setting $w \in V$ implies that $\mu _{w}:V\rightarrow  \mathfrak{g}$ is a linear map and its image $\mu _{w}\left( V\right) = A(w)(V)$ is a subspace of $\mathfrak{g}$. The following proposition shows that the coadjoint orbit for $v \in V$ is the union of subspaces $A(w)(V)$.

\begin{pps}\label{propuniafibras} 
The coadjoint orbit $\Ad_{\rho }^{\ast }\left( G_{\rho }\right) v$ of $v\in V\subset \mathfrak{g}\times V$  is given by 
\begin{equation*}
\Ad_{\rho }^{\ast }\left( G_{\rho }\right) \cdot v = \bigcup\limits_{w \in \rho \left( G \right) v} \mu_{w} \left( V\right) \times \{w \} \subset \mathfrak{g} \times V,
\end{equation*}
and writing $\mathfrak{g}\times V$ as $\mathfrak{g} \oplus V$ (with the proper identifications)
\begin{equation*}
\Ad_{\rho }^{\ast }\left( G_{\rho }\right) \cdot v = \bigcup\limits_{w \in \rho \left( G \right) v}w+A\left( w\right) \left( V\right).
\end{equation*}
\end{pps}

\begin{proof}
If $g \in G$, we can identify $\Ad_{\rho}^{\ast} (g)$ with $\rho (g)$ on $V\subset \mathfrak{g}\times V$. Therefore $\rho (g) \cdot v \subset \Ad_{\rho}^{\ast} (G_{\rho}) \cdot v$ and by (\ref{forexponencial}) given an element $w \in V \subset \mathfrak{g} \oplus V = \mathfrak{g} \times V$ $$e^{t \cdot \ad_{\rho}^{\ast} (0,v)} \cdot w = w - t\cdot A(v)(w) \quad \text{with} \quad  A(v)(w)= \mu_{w} (v).$$ Varying $v \in V$, shows that the affine subspace $w + \mu_{w}(V)$ is contained in the coadjoint orbit of $w$. Next to the fact that $\rho (G) \cdot v \subset \Ad_{\rho}^{\ast} (G_{\rho}) \cdot v$, we conclude that \begin{equation*}
\bigcup\limits_{w \in \rho \left( G\right) \cdot v } w + A(w)(V) \subset \Ad_{\rho}^{\ast}\left( G_{\rho}\right) \cdot v .
\end{equation*}
Conversely, if $g\in G$ and $w \in V$
\begin{eqnarray*}
\Ad_{\rho }^{\ast }\left( g\right) \left( w + A(w)(V) \right) &=&\rho \left( g\right) \cdot w  +\Ad_{\rho}^{\ast}\left( g\right) \mu_{w}\left( V\right) \\
&=&\rho\left( g\right) \cdot w +\mu_{\rho\left( g\right)\cdot w }\left( V\right).
\end{eqnarray*}
For $h \in G_{\rho}$ there are $g\in G$ and $\widetilde{v} \in V$ such that
\begin{equation*}
\Ad_{\rho }^{\ast }\left( h\right) \cdot v =\Ad_{\rho }^{\ast}\left( g\right) \cdot \Ad_{\rho }^{\ast}\left( e^{t\left( 0,\widetilde{v}\right)}\right) \cdot v,
\end{equation*}
as $\Ad_{\rho }^{\ast}\left( e^{t\left( 0,\widetilde{v}\right)}\right) \cdot v \in v +\mu _{v}\left( V\right) $ implies $\Ad_{\rho }^{\ast }\left( h\right) \cdot v \in \rho \left( g\right) \cdot v +\mu _{\rho \left( g\right) \cdot v }\left( V\right) $.
\end{proof}

By Proposition \ref{propuniafibras}, the coadjoint orbit $\Ad_{\rho }^{\ast }\left( G_{\rho }\right) \cdot x $, for $x \in V$ is the union of vector spaces and a fiber over $\rho \left( G\right) x$. This union is disjoint because given $Z \in \left( w +\mu_{w}\left( V\right) \right) \cap \left( v +\mu _{v}\left(V\right) \right)$, then 
\begin{equation*}
Z =w +X= v +Y\qquad X=\mu _{w}\left( x\right) ,~Y=\mu_{v}\left( w\right)
\end{equation*}
with $X,Y\in \mathfrak{g}$. Since the sum $  \mathfrak{g} \oplus V$ is direct, it follows that $w = v$ and $X=Y$. Therefore there is a fibration $\Ad_{\rho}^{\ast}\left( G_{\rho }\right) \cdot x \rightarrow \rho \left( G\right) x$ such that an element $Z =w +X\in w +\mu_{w}\left( V\right)$ associates $w \in \rho(G) x$, and whose fibers are vector spaces. The following proposition shows that this fibration is the cotangent space of $\rho \left( G\right) x$.

\begin{pps}\label{difeoorb} 
$\Ad_{\rho }^{\ast }\left( G_{\rho }\right) \cdot x $ is diffeomorphic to the cotangent bundle $T^{\ast }\left(\rho \left( G\right) x \right)$ of $\rho \left( G\right) x $, by the difeomorphism 
\begin{equation*}
\phi :\Ad_{\rho }^{\ast }\left( G_{\rho }\right) \cdot x \rightarrow T^{\ast}\left(\rho \left( G\right) x \right) ,
\end{equation*}
that satisfies 
\begin{equation*}
\phi \left( w +\mu _{w}\left( V\right) \right)
=T_{w}^{\ast}\left(\rho  \left( G\right) x \right) , \quad
w \in \rho \left( G\right) x .
\end{equation*}
The restriction of $\phi$ to a fiber $w +\mu _{w}\left( V\right)$ is given by a linear isomorphism $\mu _{w}\left( V\right) \rightarrow T_{w}^{\ast } \left(\rho \left( G\right) x \right)$.
\end{pps}

\begin{proof}
Take $X \in \Ad_{\rho }^{\ast }\left( G_{\rho }\right) \cdot x $, from the above observation there is a unique $w \in \rho \left( G\right) x$, such that $X \in \mu_{w}\left( V\right) $, then there is $v\in V$ with $X =\mu \left( v\wedge w \right)$. The vector $v\in V$ defines a linear functional $f_{v}$ on $T_{w}\left(\rho \left( G\right) x \right) $ and therefore an element of $T_{w}^{\ast }\left(\rho \left( G\right) x \right)$. Set $$\phi \left( X \right) =f_{v}\in T_{w}^{\ast }\left(\rho\left( G\right) x \right) \quad  \text{with} \quad X = w +\mu \left( v \wedge w \right).$$

An application $\phi$ is a linear injective application and the linear application $\mu \left( v\wedge w \right) \mapsto f_{v}$ is surjective. 
Furthermore, the restriction of $\phi $ to a fiber $w +\mu _{w}\left( V\right)$ is given by the isomorphism: $\mu_{w}\left( V\right) \rightarrow T_{w}^{\ast }\left(\rho \left( G\right) x \right)$. It follows that $\phi $ is a bijection.  
Finally $\phi $ is difeomorphism because both $\phi $ and $\phi ^{-1}$ are differentiable as follows by construction: $\phi$ is the identity application at the base of the bundles of $\rho\left( G\right) x$  and $\phi$ is linear on the fibers.
\end{proof}

\begin{example}[$G=\mathrm{SO}(n)$]
Take the canonical representation of $\mathfrak{g}=\mathfrak{so} \left( n\right) $ in $\mathbb{R}^{n}$. The moment map with values in $\mathfrak{g}$ is given by 
\begin{equation*}
\mu \left( v\wedge w\right) \left( B\right) =\langle Bv,w\rangle \qquad B\in 
\mathfrak{so}\left( n\right) ,
\end{equation*}
where taking $v$ and $w$ as a column vectors $n\times 1$, we have
\begin{equation*}
v\wedge w=vw^{T}-wv^{T}
\end{equation*}
which is a $n\times n$ matrix. Then
\begin{equation*}
\ad^{\ast }\left( B,v\right) =\left( 
\begin{array}{cc}
\ad \left( B\right) & -A\left( v\right) \\ 
0 & B%
\end{array}
\right) \qquad A\in \mathfrak{so}\left( n\right) ,~v\in \mathbb{R}^{n}
\end{equation*}
where for each $v\in \mathbb{R}^{n}$, $A\left( v\right) :\mathbb{R}^{n} \rightarrow \mathfrak{so}\left( n\right) $ is the application 
\begin{equation*}
A\left( v\right) \left( w\right) =v\wedge w=vw^{T}-wv^{T}.
\end{equation*}
The representation $\mathfrak{so}\left( n\right) \times \mathbb{R}^{n}$ defines a representation of the semi-direct product $G_{\rho }= \mathrm{SO}\left( n\right) \times \mathbb{R}^{n}$ on $\mathfrak{so}\left( n\right) \times \mathbb{R}^{n}$ by exponentials. As discussed earlier a $G_{\rho}$-orbit of $v\in \mathbb{R}^{n}\subset \mathfrak{s}\mathfrak{o}\left( n\right) \times \mathbb{R}^{n}$ is given by 
\begin{equation*}
\bigcup\limits_{w\in \mathcal{O}}w+A\left( w\right) \left( \mathbb{R}
^{n}\right) \qquad \mathcal{O}=\mathrm{SO}\left( n\right) \cdot v.
\end{equation*}
In this case, the orbits of $\mathrm{SO}\left( n\right)$ in $\mathbb{R}^{n} $ are the $(n-1)$-dimensional spheres centered at the origin. In particular, for $n=2$, $\mathfrak{so}\left( 2\right) \times \mathbb{R}^{2}\approx \mathbb{R}^{3}$ and for all $w$ the image $A\left( w\right) \left( \mathbb{R}^{2}\right) =\mathfrak{so}\left( 2\right)$, therefore the coadjoint orbits of the semi-direct product are the circular cylinders with axis on line determined by $\mathfrak{so}\left( 2\right)$ in $\mathfrak{so}\left( 2\right) \times \mathbb{R}^{2}\approx \mathbb{R}^{3}$.
\end{example}

In the coadjoint orbit $\Ad_{\rho }^{\ast }\left( G_{\rho }\right) \cdot x $ we can define the Kostant-Kirilov-Souriaux (KKS) symplectic form, denoted by  $\Omega $, in the same way for the cotangent bundle $T^{\ast }\left(\rho \left( G \right) x \right)$ we can define the canonical symplectic form $\omega $. 
The following proposition shows that these symplectic forms are related by the difeomorphism $\phi$ of the Proposition \ref{difeoorb}. The best way to relate these symplectic forms is through the action of the semi-direct product $G_{\rho }=G\times V$ on the cotangent bundle of $\rho \left( G\right) x $. This action is described in Proposition \ref{prop: ap mom} (in a general case), the action of $G_{\rho }$ on  $T^{\ast }\left( \rho \left( G\right) x \right) $ is Hamiltonian and therefore defines a moment 
\begin{equation*}
m:T^{\ast }\left( \rho \left( G\right) x \right) \rightarrow 
\mathfrak{g}_{\rho }.
\end{equation*}%
The construction of $m$ shows that it is the inverse of the difeomorphism $\phi $ of the Proposition \ref{difeoorb}. Moreover, $m$ is equivariant, that is, it exchanges the actions on $T^{\ast }\left( \rho \left( G\right) x \right) $ and the adjoint orbit, which implies that $m$ is a symplectic morphism. Then we conclude:

\begin{pps}
Let $\Omega$ be the symplectic form KKS in $\Ad_{\rho }^{\ast }\left( G_{\rho }\right) x $ and $\omega $ the canonical symplectic form in $T^{\ast }\left(\rho \left( G\right) x \right)$. If $\phi$ is the difeomorphism of the Proposition \ref{difeoorb}, then $\phi ^{\ast }\omega =\Omega$.  In other words, the difeomorphism $\phi$ is symplectic.
\end{pps}


\section{Adjoint orbits in semisimple Lie algebras}

Let $\mathfrak{g}$ be a non-compact semisimple (real or complex) Lie algebra  and let $G$ be a connected Lie group with finite centre and Lie algebra $\mathfrak{g}$ (for example $G=\mathrm{Aut}_{0}\mathfrak{g}$). The usual notation is:

\begin{enumerate}
    \item The Cartan decomposition: $\mathfrak{g= k \oplus s}$, with global decomposition $G=KS$.  
    \item The Iwasawa decomposition: $\mathfrak{g= k \oplus a \oplus n}$, with global decomposition $G=KAN$.
    \item $\Pi$ is the set of roots of $\mathfrak{a}$, with a choice of a set of positive roots $\Pi^{+}$ and simple roots $\Sigma \subset \Pi^{+}$ such that $\mathfrak{n}^{+} = \sum_{\alpha>0} \mathfrak{g}_{\alpha}$ and $\mathfrak{g}_{\alpha}$ is the root space of the root $\alpha$. The corresponding positive Weyl chamber is $\mathfrak{a}^{+}$. 
    \item A subset $\Theta \subset \Sigma$ defines a parabolic subalgebra $\mathfrak{p}_\Theta$ with parabolic subgroup $P_{\Theta}$ and a flag $\mathbb{F}_\Theta = G/ P_{\Theta}$. The flag is also $\mathbb{F}_{\Theta} = K/ K_{\Theta}$, where $K_{\Theta}= K \cap P_{\Theta}$. The Lie algebra of $K_{\Theta}$ is denoted by $\mathfrak{k}_{\Theta}$.
    \item Given an element $H \in \mathrm{cl}\left( \mathfrak{a}^{+} \right)$ it determines $\Theta_H \subset \Sigma$ such that $\Theta_H= \{ \alpha \in \Sigma: \alpha(H)=0 \}$. Then $\Ad(K) \cdot H = G/P_{H} = K/K_H$ is the flag denoted by $\mathbb{F}_H$ (where $P_H$ and $K_H$ denotes the subgroups $P_{\Theta_H}$ and $K_{\Theta_H}$, respectively).
    \item $b_H = 1 \cdot K_H = 1 \cdot P_H$ denotes the origin of the flag $\mathbb{F}_H$.
    \item We write $$\mathfrak{n}_{H}^{+} = \sum_{\alpha(H)>0} \mathfrak{g}_{\alpha}, \quad \mathfrak{n}_{H}^{-} = \sum_{\alpha(H)<0} \mathfrak{g}_{\alpha}$$ so that $\mathfrak{g}=\mathfrak{n}_{H}^{-} \oplus \mathfrak{z}_{H} \oplus \mathfrak{n}_{H}^{+}$, where $\mathfrak{z}_{H}$ is the centralizer of $H$ in $\mathfrak{g}$.
    \item $T_{b_H} \mathbb{F}_H \simeq \sum_{\alpha(H)<0} \mathfrak{g}_{\alpha} = \mathfrak{n}_{H}^{-}$.
    \item $Z_H = \{ g \in G: \ \Ad(g) \cdot H =H\}$ is the centralizer in $G$ of $H$. Its Lie algebra is $\mathfrak{z}_H$. Moreover, $K_H$ is the centralizer of $H$ in $K$: $$K_H = Z_{k} (H) = Z_H \cap K = \{ k \in K: \ \Ad(k) \cdot H =H\}.$$
\end{enumerate}
For more details we recommend to see \cite{helgason}, \cite{smalg} and \cite{smgrlie}.

\subsection{Semi-direct product of Cartan decomposition}\label{sdpcd}

Let $\mathfrak{g}$ be  a non-compact semisimple Lie algebra with Cartan decomposition $\mathfrak{g}=\mathfrak{k}\oplus \mathfrak{s}$. As $\left[  \mathfrak{k},\mathfrak{s}\right] \subset \mathfrak{s}$, the subalgebra $\mathfrak{k}$ can be represented on $\mathfrak{s}$ by the adjoint representation. Then, we can define the semi-direct product $\mathfrak{k}_{\ad} =\mathfrak{k}\times \mathfrak{s}$, where $\mathfrak{s}$ can be seen as an abelian algebra. This is a new Lie algebra structure on the same vector space  $\mathfrak{g}$ where the brackets $\left[ X,Y\right]$ are the same when $X$ or $Y$ are in $\mathfrak{k}$, but the bracket changes when $X, Y \in \mathfrak{s}$. The identification between $\mathfrak{k}_{\ad}=\mathfrak{k}\times \mathfrak{s}$ and its dual $\mathfrak{k}_{\mathrm{ad}}^{\ast }=\mathfrak{k}^{\ast}\times \mathfrak{s}^{\ast }$ is given by the inner product $ B_{\theta }\left( X,Y\right) =-\langle X,\theta Y\rangle $, where $\langle \cdot ,\cdot \rangle $ is the Cartan-Killing form of $\mathfrak{g}$ and $ \theta $ is a Cartan involution. If $A\in  \mathfrak{k}$, then $\ad\left( A\right) $ is anti-symmetric with respect to $ B_{\theta }$, while $\ad\left( X\right) $ is symmetric for $X\in \mathfrak{s}$. The moment map  is given by 
\begin{equation*}
\mu \left( X\wedge Y\right) \left( A\right) =B_{\theta }\left( \ad\left(
A\right) X,Y\right) \qquad A\in \mathfrak{k};~X,Y\in \mathfrak{s},
\end{equation*}
the second part of that equality is 
\begin{equation*}
B_{\theta }\left( \left[ A,X\right] ,Y\right) =-B_{\theta }\left( \left[ X,A \right] ,Y\right) =-B_{\theta }\left( A,\left[ X,Y\right] \right) =-\langle A,\left[ X,Y\right] \rangle
\end{equation*}
because $\left[ X,Y\right] \in \mathfrak{k}$. Therefore the moment map of the adjoint representation of $\mathfrak{k}$ on $\mathfrak{s}$ is 
\begin{equation*}
\mu \left( X\wedge Y\right) =\left[ X,Y\right] \in \mathfrak{k}\qquad X,Y\in \mathfrak{s},
\end{equation*}%
where $\left[ \cdot ,\cdot \right] $ is the usual bracket of $\mathfrak{g}$. Therefore, the coadjoint representation of the semi-direct product $\mathfrak{k}\times \mathfrak{s}$ is given by (in an orthonormal basis) 
\begin{equation*}
\ad^{\ast }\left( X,Y\right) =\left( 
\begin{array}{cc}
\ad\left( X\right) & -A\left( Y\right) \\ 
0 & \ad (X)
\end{array}%
\right) \qquad X\in \mathfrak{k}, \ Y\in \mathfrak{s}  \label{forAdEstrela}
\end{equation*}%
where for each $Y\in \mathfrak{s}$, $A\left( Y\right) :\mathfrak{s}\rightarrow \mathfrak{k}$ is the map $A\left( Y\right) \left( Z\right) =\left[ Y,Z\right] $. 

Let  $K\subset G$ the subgroup given by $K=\langle \exp  \mathfrak{k}\rangle $. The semi-direct product of $K$ in $\mathfrak{s}$ will be denoted by $K_{\ad}=K\times \mathfrak{s}$.
The coadjoint orbit of $\widetilde{X}\in \mathfrak{s}\subset \mathfrak{k}\times  \mathfrak{s}$ is the union of the fibers $A\left( Y\right) \left( \mathfrak{s}\right) $ with $Y$ passing through the $K$-adjoint orbit of $\widetilde{X}$ in $\mathfrak{s}$.  As $A\left( Y\right) \left( Z\right) =\left[ Y,Z\right]$, then $A\left( Y\right) \left( \mathfrak{s}\right) =\ad\left( Y\right) \left( \mathfrak{s}\right)$ where $\ad$ is the adjoint representation in $\mathfrak{g}$.

To detail the coadjoint orbits of the semi-direct product, take a maximal abelian subalgebra $\mathfrak{a}\subset \mathfrak{s}$. The $\Ad\left( K\right) $-orbits in $\mathfrak{s}$ are passing through $\mathfrak{a}$, thus are the flags on $\mathfrak{g}$. Take a positive Weyl chamber $\mathfrak{a}^{+}\subset \mathfrak{a}$, if $H\in \mathrm{cl}\left( \mathfrak{a}^{+} \right)$ then the orbit $\Ad \left( K\right) H$ is the flag manifold $\mathbb{F}_{H}$.  By Proposition \ref{difeoorb}, the $K_{\ad}$-orbit in $H\in \mathrm{cl}(\mathfrak{a}^{+})$ is diffeomorphic to the cotangent bundle of $\mathbb{F}_{H}$, thus the $K_{\ad}$-orbit itself is the union of the fibers $\ad \left( Y\right) \left( \mathfrak{s}\right)$, with $Y\in \mathbb{F}_{H}$. In this union the fiber over $H$ is $H+\ad \left( H\right) \left( \mathfrak{s }\right)$ with $\ad \left( H\right) \left( \mathfrak{s} \right) \subset  \mathfrak{k}$. With the notations above this subspace of $\mathfrak{k}$ is given by 
\begin{equation*}
\ad \left( H\right) \left( \mathfrak{s}\right) =\sum_{\alpha \left( H\right)
>0}\mathfrak{k}_{\alpha }.
\end{equation*}

\begin{example}\label{sl2}
Take $\mathfrak{sl}\left( 2,\mathbb{R}\right) $ with basis $\{H,S,A\}$ given by 
\begin{equation*}
H=\left( 
\begin{array}{cc}
1 & 0 \\ 
0 & -1%
\end{array}
\right), \qquad S=\left( 
\begin{array}{cc}
0 & 1 \\ 
1 & 0%
\end{array}%
\right), \qquad A=\left( 
\begin{array}{cc}
0 & 1 \\ 
-1 & 0%
\end{array}%
\right),
\end{equation*}%
and coordinates $\left( x,y,z\right) =xH+yS+zA$.  The Cartan decomposition $\mathfrak{k}\oplus \mathfrak{s}$ is given by $\mathfrak{k}= \mathfrak{so}\left( 2\right) =\langle A\rangle $ and $\mathfrak{s}=\langle H,S\rangle $. The adjoint representation of $\mathfrak{so}\left( 2\right)$ in $\mathfrak{s}$ coincides with its canonical representation in $\mathbb{R}^{2}$. Hence,  the coadjoint orbits of the semi-direct product are the cylinders $x^{2}+y^{2}=r$, with $r>0$ and on the $z$-axis (generated by $A$) the orbits degenerate into points.
\end{example}


\section{Deformations of Lie algebras} \label{deforb}

Let $\mathfrak{g}$ be  a non-compact semisimple Lie algebra, we will provide  a compatible structure that allows us to deform the adjoint orbit of $G$ in the coadjoint orbit of $K_{\ad}$. As previously stated, the idea is based on case $\mathfrak{sl}(2,\mathbb{R})$ (Example \ref{sl2}), where we obtain a cylinder as an orbit, while in the usual case the result is an hyperboloid.

Fixing $\mathfrak{g}=\mathfrak{k}\oplus  \mathfrak{s}$ a Cartan decomposition, with $\theta$ its Cartan involution, the two structures of Lie algebras for $\mathfrak{g}$ (semisimple and semi-direct product) give rise to different coadjoint orbits. In both cases, the orbits that pass through any element of $\mathfrak{s}$ are diffeomorphic to the cotangent bundles of the flags of $\mathfrak{g}$, and therefore diffeomorphic to each other. 

For this, consider $r>0$ and define the linear map 
\begin{equation*}
T_{r}:\mathfrak{g}\rightarrow \mathfrak{g}\quad \text{ such that }\quad
T_{r}(X)=rX\quad \forall X\in \mathfrak{k}\quad \text{ and }\quad
T_{r}(Y)=Y\quad \forall Y\in \mathfrak{s},
\end{equation*}%
and induce the Lie bracket 
\begin{equation*}
\left[ X,Y\right] _{r}=T_{r}\left[ T_{r}^{-1}X,T_{r}^{-1}Y\right] ,
\label{eq: adr}
\end{equation*}
such that $(\mathfrak{g}, [ \cdot , \cdot ]_r)$ is a Lie algebra. In general we have:

\begin{lem}\label{isoalg}
For $r>0$, denote by $\mathfrak{g}_r$ the Lie algebra  $(\mathfrak{g}, [ \cdot , \cdot ]_r)$ and by $\langle \cdot , \cdot \rangle_r$ its Cartan-Killing form. Then

\begin{enumerate}
    \item $T_r: \mathfrak{g} \rightarrow \mathfrak{g}_r$ is an isomorphism of Lie algebras.
    \item $\langle X,Y\rangle_{r} = \langle T_{r}^{-1}X,T_{r}^{-1} Y \rangle$ for all $X,Y \in \mathfrak{g}$.
    \item $\mathfrak{g}_r$ supports the same Cartan decomposition of $\mathfrak{g}$.
\end{enumerate}
\end{lem}

\begin{proof} The isomorphism between that Lie algebras is also immediate because 
\begin{equation*}
\left[ T_r X,T_r Y\right]_{r}=T_r \left[ T_{r}^{-1}T_r X,T_{r}^{-1} T_r Y\right] =T_r\left[ X,Y\right],
\end{equation*}
that is, $T_r$ is a homomorphism of Lie algebras. In relation to the Cartan-Killing form, as $T_r$ is an isomorphism, then  
\begin{equation*}
\ad_{r}\left( X\right) =T_r \circ \ad \left( T_{r}^{-1} X \right) \circ T_{r}^{-1}.
\end{equation*}
Hence 
\begin{equation*}
\langle X,X\rangle_{r}=\tr\left( \ad_{r}\left( X\right) \right) ^{2}=\tr \left( \ad \left( T_{r}^{-1}X\right) \right) ^{2}=\langle T_r^{-1}X,T_r^{-1}X\rangle
\end{equation*}
showing the second statement. For the last statement, if $\theta$ is a Cartan involution corresponding to the decomposition $\mathfrak{g = k \oplus s}$, then $\widetilde{\theta}= T_r \circ \theta \circ T_{r}^{-1}$ is a Cartan involution that satisfies $$\mathfrak{k}=\{ X \in \mathfrak{g}_r: \widetilde{\theta}X =X \} \quad \text{and} \quad \mathfrak{s}=\{ Y \in \mathfrak{g}_r: \widetilde{\theta}Y =-Y \},$$ that is, $\mathfrak{g}_r$ has the same Cartan  decomposition.
\end{proof}

\begin{remark}
As a consequence of the Lemma \ref{isoalg}, in $\mathfrak{g}_r$ we can choose the same maximal abelian algebra $\mathfrak{a \subset s}$ and the same simple root system $\Sigma$ (consequently the same root system $\Pi$ and positive Weyl chamber $\mathfrak{a}^{+}$), but the root spaces are going to change.
\end{remark}

Denote by $\ad_r: \mathfrak{g}_r \rightarrow \mathfrak{g}_r$ the $r$-adjoint representation, given by $$\ad_r (X) (Y) = [X,Y]_r \quad X,Y \in \mathfrak{g}_r.$$
Then for $\alpha \in \Pi$ the corresponding $r$-root space is defined by 
\begin{equation*}
    \mathfrak{g}_{\alpha}^{r} = \{ X \in \mathfrak{g}_r: \ \ad_r (H)X = \alpha(H) X \quad \forall H \in \mathfrak{a} \}.
\end{equation*}
If $\mathfrak{g}_{\alpha}$ is the usual root space (of $\mathfrak{g}$). Then we can define the application $\psi _r: \mathfrak{g} \rightarrow  \mathfrak{g}$ given by
\begin{equation*}  \label{gammar}
\psi_r (Z)  = Z + \frac{r-1}{r+1}\theta Z,
\end{equation*}  
such that $\mathfrak{g}_\alpha ^{r} = \psi_r \left( \mathfrak{g}_\alpha \right)$ for all $\alpha \in \Pi$. So if $\mathfrak{n}^{+} = \sum_{\alpha >0} \mathfrak{g}_\alpha$, we have: 
\begin{equation*}
\mathfrak{n}_{r}^{+}= \sum_{\alpha >0} \mathfrak{g}_\alpha ^{r} = \sum_{\alpha >0} \psi_r \left( \mathfrak{g}_\alpha \right) = \psi_r \left( \mathfrak{n}^{+} \right),
\end{equation*}
thus given $H \in \mathrm{cl}(\mathfrak{a}^{+})$
\begin{equation*}  
\mathfrak{n}_H ^{+} = \sum_{\alpha(H) >0} \mathfrak{g}_\alpha \quad \text{and} \quad \mathfrak{n}_{r,H} ^{+} = \psi_r \left( \mathfrak{n}_H ^{+} \right).
\end{equation*}

Let be $G_r = \mathrm{Aut}_{0}\mathfrak{g}_r$, that is, $ G_r$ is semisimple and diffeomorphic to $G$, whose adjoint orbits are manifolds in $\mathfrak{g}$.  The $r$-adjoint representation of $G_r$ is going to be defined (identified) by: 
\begin{equation*} 
\Ad_r (G) \cdot H = \Ad_r (K) \left( H + \psi_r \left( \mathfrak{n}_H ^{+} \right) \right).
\end{equation*}
Now, let's see some results that will allow to describe the $r$-adjoint orbit: 

\begin{lem}
For $r>0$, $X \in \mathfrak{n}$ and $Y \in \mathfrak{g}$, then 
\begin{equation*}
\Ad_r \left( e^{tX} \right) \cdot Y = \Ad \left( e^{\frac{t}{r}X} \right)
\cdot Y.
\end{equation*}
\end{lem}

\begin{proof}
For $r>0$, $X\in \mathfrak{k}$ and $Y \in \mathfrak{g}$, then 
\begin{equation*}
    \ad_r (X) \cdot Y = T_r \left[\frac{1}{r} X, \frac{1}{r} \kappa(Y) + \sigma (Y) \right] =\frac{1}{r} \ad(X) \cdot Y.
\end{equation*}
Inductively $\ad_r ^k (X) \cdot Y = \frac{1}{r^k} \ad ^k (X) \cdot Y$, then 
\begin{equation*}
\sum_{k>0} \cfrac{t^k \ad_r ^k (X)}{k!} Y = \sum_{k>0} \cfrac{
\left(\frac{t}{r}\right) ^k \ad ^k (X)}{k!} Y = \Ad \left( e^{\frac{t}{r} X }\right) \cdot Y. \qedhere
\end{equation*}
\end{proof}

Therefore, given $H\in \mathrm{cl}\left( \mathfrak{a}^{+}\right) $ and $X\in \mathfrak{k}$: 
\begin{equation*}
\Ad_{r}\left( e^{tX}\right) \cdot H =\Ad\left( e^{\frac{t}{r}X}\right) \cdot H,
\end{equation*}
that is, they determine the same flag manifold. In addition, given $\alpha \in \Pi$ 
\begin{align*}
\Ad_r \left( e^{tX} \right) \cdot X_\alpha ^{r} &= e^{\frac{t}{r} \ad(X)} \left( \psi_r \left( X_\alpha \right) \right) \\
&= \sum_{k>0} \cfrac{\left( \frac{t}{r} \right)^{k} \ad ^{k} (X)}{k!} \left( \psi_r \left( X_\alpha \right) \right) \\
&= \psi_r \left( \sum_{k>0} \cfrac{\left( \frac{t}{r} \right)^{k} \ad ^{k} (X)}{k!} \left( X_\alpha \right) \right) = \psi_r \left( \Ad \left( e^{\frac{t}{r} X }\right) \cdot X_\alpha \right).
\end{align*}
Thus for all $k \in K$
\begin{equation}  \label{comut}
\Ad_r (k) \left( \mathfrak{n}_{r,H} ^{+} \right)= \Ad_r (k) \cdot \psi_r \left( \mathfrak{n}_{H} ^{+} \right)= \psi_r \left( \Ad (k) \left(  \mathfrak{n}_{H} ^{+} \right) \right)
\end{equation}
Hence,  we conclude:

\begin{pps}
\label{teo: deformac} For $r>0$ and $H\in \mathrm{cl}\left( \mathfrak{a}^{+}\right)$ 
\begin{equation*}
\Ad_r (G) \cdot H = \underset{k \in K}{\bigcup} \Ad (k) \left( H \right) + \psi_r \left( \Ad (k) \left( \mathfrak{n}_H ^{+} \right) \right),
\end{equation*}
that is, the $r$-adjoint orbit of $G$ is a $r$-deformation of the adjoint orbit of $G$.
\end{pps}

\begin{remark}
The construction of $\psi_r$ is given by the fact that there are $X_{\alpha} \in \mathfrak{g}_{\alpha}$  such  that $\theta X_{\alpha}= - X_{-\alpha} \in \mathfrak{g}_{-\alpha}$ for all $\alpha \in \Pi$.  Then given $H \in \mathfrak{a \subset s}$:  
\begin{equation*}
    \ad_r (H) (\psi_r X_\alpha )= \ad_r(H)(X_{\alpha}) + \frac{r-1}{r+1}\ad_r(H)(X_{- \alpha}),  
\end{equation*}
and
\begin{equation*}
\ad_r (H) X_\alpha = \alpha (H) \left( \frac{1 + r^2}{2r} X_\alpha - \frac{ r^2 -1}{2r} X_{-\alpha} \right).
\end{equation*}
Therefore $\ad_r (H) \left( X_{\alpha}^r \right) = \alpha (H) \left( X_\alpha + \frac{1-r}{r+1} X_{-\alpha} \right) = \alpha (H) \left( X_{\alpha}^r \right)  $.
\end{remark}

In addition, the representation of $K_H$ in $\mathfrak{g}_r$ makes invariant the subspace $\mathfrak{n}_{r, H} ^{+} $, because if $k \in K$, then $\Ad_r (k)$ commutes with $\ad_r \left( H \right)$. Therefore $\Ad_r (k)$ takes eigenspaces of $\ad_r (H)$ in eigenspaces. Thus we can induce the representation $\rho_r$ of $K_H$ in $\mathfrak{n}_{r,H} ^{+}$,  and by (\ref{comut}), we have 
\begin{equation*}  
 \psi_r \left( \rho (k) \cdot X \right) = \rho_r (k) \cdot \psi_r (X)  \qquad k \in K_H, \ X \in \mathfrak{n}_H ^{+}, 
\end{equation*}
where $\rho$ is the representation in the case $r=1$ (that is, the usual representation $\Ad$). Therefore 
\begin{equation*}  
K\times_{\rho_r} \mathfrak{n}_{r,H} ^{+} =  K\times_{\rho}  \psi_r \left( \mathfrak{n}_{H} ^{+} \right) .
\end{equation*}

So let's induce a diffeomorphism between $\Ad (G) \cdot H$ and $\Ad_r (G) \cdot H$ using the following map (this construction was proved in \cite[Proposition 2.4]{smgasgr1})
\begin{equation*}
\gamma_r : \Ad_r (G) \cdot H \rightarrow K\times_{\rho}  \psi_r \left( \mathfrak{n}_{H} ^{+} \right),
\end{equation*}
such that 
\begin{equation*}
Y= \Ad_r (k) (H + X) \mapsto  (k, X) \in  K\times_{\rho}  \psi_r \left( \mathfrak{n}_H ^{+} \right)
\end{equation*}
is a diffeomorphism that satisfies:

\begin{enumerate}
\item $\gamma_r$ X is equivariant with respect to the action of $K$.

\item $\gamma_r$ leads fibers into fibers.

\item $\gamma_r$ leads the orbit $\Ad_r (K) \cdot H$ in the null section of  $K\times_{\rho}  \psi_r \left( \mathfrak{n}_{H} ^{+} \right) $.
\end{enumerate}

It's easy to see that $( d\gamma_r)_x = \id$, for $x= \Ad_r (k) (H + Y) \in \Ad_r (G) \cdot H$.

Furthermore, difeomorphism $\gamma_r$ is defined by the vector bundle $K\times_{\rho}  \psi_r \left( \mathfrak{n}_{H} ^{+} \right)$ associated with the main bundle  $K \rightarrow K / K_H$ of $\Ad_r (G) \cdot H$, which is a homogeneous space. Using this diffeomorphism for $r>0$ we define the map $\widetilde{\psi}_r$ as follows: 
\begin{equation*}
\begin{tikzcd} \Ad(G) \cdot H \arrow{rr}{\widetilde{\psi}_r}
\arrow[swap]{d}{\gamma} & & \Ad_r (G) \cdot H \arrow{d}{\gamma_r}
\\K\times_{\rho} \mathfrak{n}_{H} ^{+} \arrow{rr}{\psi_r}& &
K\times_{\rho}  \psi_r \left( \mathfrak{n}_{H} ^{+} \right) \end{tikzcd}
\end{equation*}
which is a diffeomorphism, because $\psi_r$ is linear (in the complex is the sum of linear and anti-linear applications) and $\gamma_r$ is a difeomorphism, as seen above. We conclude 
\begin{equation*}
\left( d\widetilde{\psi} \right)_x = \psi_r, \qquad x \in \Ad_r (G) \cdot H.
\end{equation*}

Hence joining these constructions, we conclude that

\begin{teo}\label{difeor}
Let $H \in \mathrm{cl}(\mathfrak{a}^{+})$ and $r>0$, then the manifolds $\Ad (G) \cdot H$ and $\Ad_r (G) \cdot H$ are diffeomorphic by $\widetilde{\psi}_r$.
\end{teo}

In addition, our intention is to take the above diffeomorphism to a diffeomorphism between $\Ad_r (G) \cdot H$ and $ K_{\ad} \cdot H$, for this define the application 
\begin{equation*}
\psi(X) = X + \theta X,
\end{equation*}
and notice that when $r \rightarrow \infty$
\begin{equation*}
\psi_r \rightarrow \psi.
\end{equation*}

\begin{lem}
\label{constpsi} The application $\psi$ defined above satisfies:

\begin{enumerate}
\item The image of $\psi$ is in $\mathfrak{k}$ and the kernel in $\mathfrak{s}$.

\item Let $X \in \mathfrak{k}$, we have $\psi \circ \ad (X) = \ad (X) \circ \psi$.

\item If $X \in \mathfrak{k}$ and $\alpha \in \Pi$, then 
\begin{equation*}
\Ad \left( e^{tX} \right) \psi X_\alpha = \psi \left( \Ad \left( e^{tX} \right) X_\alpha \right).
\end{equation*}
\end{enumerate}
\end{lem}

\begin{proof}
Item 1 is immediate from the definition of $\kappa$.
\begin{enumerate}
\item[\emph{2. }] Let $X \in \mathfrak{k}$ and $Y \in \mathfrak{g}$ 
\begin{equation*}
    \psi [X,Y] = [X,Y] + \theta [X,Y] =  [X,Y] + [X, \theta Y] = [X, \psi Y].
\end{equation*}

\item[\emph{3. }] Let $X \in \mathfrak{k}$ and $\alpha \in \Pi$, note that
for $Y\in \mathfrak{g}$ 
\begin{equation*}
\ad(Y) \cdot \theta X_\alpha = [Y,\theta X_\alpha ] = \theta [\theta Y,
X_\alpha ] = \theta \ad(\theta Y) X_\alpha,
\end{equation*}
inductively, we have $\ad^{k}Y \cdot \theta X_\alpha = \theta \ad^{k}(\theta
Y) \cdot X_\alpha$, then 
\begin{align*}
e^{t \ad(X)} \left( \theta X_\alpha \right) &= \sum_{k>0} \cfrac{ t^{k} \ad
^{k} (X)}{k!} \cdot \theta X_\alpha \\
&= \sum_{k>0} \cfrac{ t^{k} \theta \ad ^{k} (\theta X)}{k!} \cdot X_\alpha = \theta \cdot e^{t \ad(X)} \left( X_\alpha \right),
\end{align*}
because $\theta X = X$ and we have 
\begin{align*}
\Ad \left( e^{tX} \right) \cdot \psi X_\alpha &= e^{t \ad (X)} \left(
X_\alpha \right) + \theta \cdot e^{t \ad (X)} \left( X_\alpha \right) \\
&= \psi \left( e^{ t \ad(X)} \cdot X_\alpha \right) = \psi \left( \Ad \left( e^{t X }\right) X_\alpha \right). \qedhere
\end{align*} 
\end{enumerate}
\end{proof}

Hence define 
\begin{equation}
\Ad_{\infty} (G) \cdot H : = \underset{k \in K}{\bigcup} \Ad (k) \left(
H  + \psi \left( \mathfrak{n}_H ^{+}
\right) \right),
\end{equation}
when $r \rightarrow \infty$ we have 
\begin{equation*}
\Ad_{r} (G) \cdot H \rightarrow \Ad_{\infty} (G) \cdot H.
\end{equation*}

It is convenient to define the $\infty$-root spaces by $\mathfrak{g}_\alpha ^{\infty} = \psi \left( \mathfrak{g}_\alpha \right),$ and consequently 
\begin{equation*} 
\mathfrak{n}_{\infty , H} ^{+} = \sum_{\alpha(H) >0} \mathfrak{g}_\alpha ^{\infty} = \sum_{\alpha (H)>0 } \psi \left( \mathfrak{g}_\alpha \right) = \psi \left( \mathfrak{n}_H ^{+} \right) .
\end{equation*}

Therefore, analogous to the above, define 
\begin{equation*}
\gamma_\infty : \Ad_\infty (G) \cdot H \rightarrow K\times_{\rho} \psi \left( \mathfrak{n}_H ^{+} \right) ,
\end{equation*}
such that 
\begin{equation*}
Y= \Ad (k) (H + X ) \mapsto  (k, X ) \in K\times_{\rho}  \psi (\mathfrak{n}_{H} ^{+}) .
\end{equation*}

So we have that the application $\gamma_\infty$ is a diffeomorphism, seen as a vector bundle. The map $\gamma_\infty$ is well defined as a consequence of $\psi$, the bijectivity is a consequence of the way in which the manifold $\Ad(G)_{\infty} \cdot H$ was defined, and the differentiability is given by the idea of making $r \rightarrow \infty$, in the difeomorphism $\gamma_r$. Then we can define the  difeomorphism $\widetilde{\psi}$, given by: 
\begin{equation*}
\begin{tikzcd} \Ad(G) \cdot H \arrow{rr}{\widetilde{\psi}}
\arrow[swap]{d}{\gamma} & & \Ad_{\infty} (G) \cdot H
\arrow{d}{\gamma_\infty} \\K\times_{\rho} \mathfrak{n}_{H} ^{+}
\arrow{rr}{\psi}& & K\times_{\rho} \psi \left( \mathfrak{n}_H ^{+} \right)
\end{tikzcd}
\end{equation*}
in the same way as for $\widetilde{\psi}_r$. So using $r \rightarrow \infty$ 
\begin{equation*}
\left( d\gamma_\infty \right)_x = \id \quad \text{and} \quad \left( d 
\widetilde{\psi} \right)_x = \psi.
\end{equation*}

Soon, we can conclude that

\begin{teo}\label{difeoinf}
The manifolds $\Ad_{\infty} (G) H$ and $\Ad_{r} (G) \cdot H$ are diffeomorphic for $r>0$, the diffeomorphisms are given by $\widetilde{\psi}$ and $\widetilde{\psi}_r$ defined above.
\end{teo}

Then 
\begin{equation*}
K_{\ad}\cdot H =  
\underset{k \in K}{\bigcup} \Ad (k) \left( H \right) + \left[ \Ad (k) \cdot H , \mathfrak{s} \right].
\end{equation*}
The fiber in $H$ is $H + [H , \mathfrak{s}]$, but $\mathfrak{s}= \mathfrak{a} \oplus \sigma (\mathfrak{n})$, then $[H , \mathfrak{s}] = [H , \sigma (\mathfrak{n})]$ because $[H ,  \mathfrak{a}]=0$. In addition, for $\alpha \in \Pi ^{+}$ 
\begin{equation*}
[H , X_\alpha] = \alpha (H ) X_\alpha,
\end{equation*}
such that

\begin{itemize}
\item $\alpha (H) =0$ if $\alpha \notin \langle \Theta_H \rangle ^{+}$, then $\alpha (H) \mathfrak{g}_\alpha = 0$ for $\alpha \notin \langle \Theta_H \rangle ^{+}$.

\item $\alpha (H) > 0$ if $\alpha \in \langle \Theta_H \rangle ^{+}$, then $\alpha (H) \mathfrak{g}_\alpha = \mathfrak{g}_\alpha$ for $ \alpha \in \langle \Theta_H \rangle ^{+}$.
\end{itemize}

Thus $[H , \sigma (\mathfrak{n})]= \frac{1}{2} \left( [H , X_\alpha] - [H , \theta X_\alpha] \right)$, but 
\begin{equation*}
[H , \theta X_\alpha] = \theta [\theta H , X_\alpha] = \theta [- H , X_\alpha] = - \theta [ H , X_\alpha],
\end{equation*}
because $H \in \mathfrak{a}\subset \mathfrak{s}$, then 
\begin{equation*}
  [H , \sigma (X_\alpha )] = \frac{1}{2} \left( [H , X_\alpha] + \theta [H , X_\alpha] \right)  = \frac{1}{2} \psi \left( \alpha (H) \cdot X_\alpha \right),
  \end{equation*}
and as $[H , X_\alpha] \neq 0$ if and only if $\alpha \in \langle \Theta_H \rangle ^{+}$, we have
\begin{equation*}
[H , \mathfrak{s}] = \psi \left( \mathfrak{n}_H ^{+} \right).
\end{equation*}
So the fibers in $H$ of $K_{\ad}\cdot H$ and $\Ad_\infty (G) \cdot H$ coincide. In an equivalent way, we can identify the other fibers of these spaces for each $k \in K$. 
\begin{equation*}
\Ad (k) \left( H \right) + \underset{\in \mathfrak{k}}{\underbrace{  \left[ \Ad (k) \cdot H , \mathfrak{s} \right]}} \mapsto \Ad (k) \left( H \right) + \underset{\in \mathfrak{k}}{\underbrace{ \psi \left( \Ad (k) \cdot \mathfrak{n}_H ^{+} \right)}} .
\end{equation*}
So we can identify the manifolds $K_{\ad}\cdot H$ and $\Ad_\infty (G) \cdot H $ that are diffeomorphic from the bundle $T^{\ast}  \mathbb{F}_H$. We can conclude:

\begin{cor}
\label{cororb} The adjoint orbit $\Ad(G) \cdot H$ deforms in $K_{\ad}\cdot H$, by $\psi _{r}$.
\end{cor}


\section{Hermitian symplectic form} \label{csalg}

In this section we will take advantage of the construction given in Section  \ref{csalg} for  complex semisimple algebras. Let  $\mathfrak{g}$ be a complex semisimple algebra and $\mathfrak{u}$ the compact real form of $\mathfrak{g}$, such that $\mathfrak{g}=\mathfrak{u}\oplus i\mathfrak{u}$ is a Cartan decomposition of $\mathfrak{g}$, with Cartan involution $\tau$.  Let $G = \mathrm{Aut}_{0}\mathfrak{g}$, then for $X,Y \in \mathfrak{g}$ 
\begin{equation*}
\mathcal{H}_{\tau}(X,Y) = - \langle X, \tau Y \rangle 
\end{equation*}
is a Hermitian form of $\mathfrak{g}$, where $\langle \cdot , \cdot \rangle$ is the complex Cartan-Killing form of $\mathfrak{g}$ (see \cite[Lema 12.17]{smalg}). The imaginary part of $\mathcal{H}_\tau$ will be denoted by $\Omega_\tau$, and $\Omega_\tau$ is a symplectic form on $\mathfrak{g}$. In addition, for $H \in \mathfrak{g}$ the restriction of $\Omega_\tau$ in the adjoint orbit $\Ad (G) \cdot H$ is a symplectic form.  Furthermore,  the restriction of $\Omega_\tau$ in $\Ad_r (G) \cdot H$  is a symplectic form for $r>0$, because $\mathcal{H}_\tau$ is given by the Cartan-Killing form of $\mathfrak{u}$ which is $\Ad(U) $-invariant.

\subsection{Lagrangian sections} \label{secLagran}

As $\mathfrak{g}=\mathfrak{u}\oplus i\mathfrak{u}$ is a Cartan decomposition with Cartan involution $\tau$, for $\mathfrak{g}$ semisimple complex Lie algebra. If $U\subset G$ is the compact subgroup with Lie algebra $\mathfrak{u}$. Then, the representation of the semi-direct product (described above in the general case) is $U_{\ad}$. If $H\in \mathfrak{s}=i\mathfrak{u}$, its semi-direct orbit is denoted by $U_{\ad} \cdot H$. To begin with, let's see that the restriction of  $\Omega_\tau ( \cdot , \cdot )= \mathrm{im} \left(\mathcal{H}_{\tau} (\cdot , \cdot ) \right)$ is a symplectic form in $U_{\ad} \cdot H$.

\begin{pps}
\label{propRestriNaoDegene} The form $\Omega_\tau$ of $\mathfrak{g}$ restricted to $U_{\ad} \cdot H$ is a symplectic form, for $H \in \mathrm{cl}(\mathfrak{a}^{+})$.
\end{pps}

\begin{proof}
The restriction is a closed 2-form because it is the pull-back of the imaginary part of $\mathcal{H}_{\tau }$ by inclusion. Hence, it remains to be seen that the restriction is a non-degenerate 2-form. Take a semi-direct coadjoint orbit
\begin{equation*}
\mathcal{O} =\bigcup\limits_{Y\in \Ad \left( U\right) H}\left( Y+\ad \left( Y\right) \left( i\mathfrak{u}\right) \right), \qquad H\in \mathrm{cl} \left( \mathfrak{a}^{+} \right).
\end{equation*}

The  tangent space to a fiber $Y+\ad \left( Y\right) \left( i \mathfrak{u}\right)$ is $\ad \left( Y\right) \left( i\mathfrak{u} \right)$ which is a subspace of $\mathfrak{u}$, and a Lagrangian subspace of $\mathfrak{g}$. Hence the tangent spaces to the fibers are isotropic subspaces for the  restriction of $\Omega_\tau$. The dimension of a fiber is half the dimension of the total orbit. Therefore, by Proposition \ref{propFormaDegenerada} to prove that the restriction of $\Omega_\tau$ is non-degenerate, it is enough to show that the tangent spaces to fibers are maximal isotropic. Take an element $\xi =H+X$ in the fiber over the origin $H$ with $X\in \ad \left( H\right) \left( i\mathfrak{u}\right)$. In terms of root spaces
\begin{equation*}
\ad \left( H\right) \left( i\mathfrak{u}\right) =\sum_{\alpha \in \Pi} \mathfrak{u}_{\alpha }
\end{equation*}
where $\mathfrak{u}_{\alpha }=\left( \mathfrak{g}_{\alpha }\oplus \mathfrak{g}_{-\alpha }\right) \cap \mathfrak{u}$. The tangent space $T_{\xi} \mathcal{O}$ of the  orbit $\mathcal{O}$ in $\xi =H+X$ is generated by this vertical space $\ad \left( H\right) \left( i\mathfrak{u}\right)$ and by the vectors $\ad \left( A \right) \xi$, with $A\in \mathfrak{u}$, such that
\begin{equation*}
\left[ A,H+X\right] =\left[ A,H\right] +\left[ A,X\right] \qquad X\in \ad\left( H\right) \left( i\mathfrak{u}\right) \subset \mathfrak{u}.
\end{equation*}
The component $\left[ A,X\right] \in \mathfrak{u}$, so if $v\in \ad \left( H\right) \left( i\mathfrak{u}\right)$ is a vector of the vertical tangent space then
\begin{equation*}
\Omega_\tau \left( v,\left[ A,H\right] +\left[ A,X\right] \right) =\Omega_\tau \left( v,\left[ A,H\right] \right)
\end{equation*}
since the Hermitian form $\mathcal{H}_{\tau}$ is real in $\mathfrak{u}$, that is, $\mathfrak{u}$ is a Lagrangian subspace for $\Omega_\tau $. Then, to show that the tangent space to the fiber is maximal isotropic it must be shown that given $\left[ A,H\right]$ with $A\in \mathfrak{u}$, there is an element $v$ of the tangent space to the fiber such that $\Omega_\tau \left( v,\left[ A,H\right] \right) \neq 0$. Now, the subspace 
\begin{equation*}
\{\left[ A,H\right] :A\in \mathfrak{u}\}
\end{equation*}
is nothing less than the tangent space to the  orbit $\Ad \left( U\right) \cdot H$ and is given by $\ad \left( H\right) \left( \mathfrak{u}\right) =i\ad \left( H\right) \left( i\mathfrak{u}\right)$. Therefore, it all comes down to verify that given $Z\in \ad \left( H\right) \left( \mathfrak{u}\right) \subset \mathfrak{s}=i\mathfrak{u}$, $Z\neq 0$, there is $v\in \ad \left( H\right) \left( i\mathfrak{u}\right)$ such that $\Omega_\tau \left( v,Z\right) \neq 0$. But this is immediate as $\mathcal{H}_{\tau} \left( Z,Z\right) >0$ since $\mathcal{H}_{\tau}$ is positively defined in $\mathfrak{s}$. Hence if $v=iZ\in \ad \left( H\right) \left( i\mathfrak{u}\right) $ then \begin{equation*}
\mathcal{H}_{\tau }\left( v,Z\right) =\mathcal{H}_{\tau }\left( iZ,Z\right)
=i\mathcal{H}_{\tau }\left( Z,Z\right)
\end{equation*}
is imaginary and $\neq 0$ which means that $\Omega_\tau \left( v,Z\right) \neq 0$.

In short, it was shown that (the restriction of) $\Omega_\tau$ is a symplectic form along the fiber $H+\ad \left( H\right) \left( i\mathfrak{u}\right)$ over the origin $H$. In the other fibers the result is obtained by using the fact that $\Omega_\tau$ is invariant by $U$ and taking into account that the fiber over $Y=\Ad \left( u\right) \cdot H$, $u\in U$, is given by $\Ad \left( u\right) \left( H+\ad \left( H\right) \left( i\mathfrak{u}\right) \right)$.
\end{proof}

Then the restriction of $\Omega_\tau$ in the coadjoint orbit will also be denoted by $\Omega_\tau$. In addition, the restriction of $\mathcal{H}_{\tau} $ to $\mathfrak{s}= i\mathfrak{u}$ is the Cartan-Killing form $\langle \cdot ,\cdot \rangle $, which is an inner product in $\mathfrak{s}$ and induces an $U$-invariant Riemannian metric in an orbit $\Ad \left( U\right) \cdot H$.

\begin{pps}
Given $Y\in \mathfrak{s}$ and $Z\in \Ad \left( U\right) \cdot H$, suppose that $Y\in T_{Z} \Ad \left( U\right) \cdot H$. Then $iY \in \ad\left(Z\right) \left( i\mathfrak{u}\right)$, that is, $Z+iY$ is in the fiber over $Z$ of the semi-direct coadjoint orbit.
\end{pps}

\begin{proof}
Take first $Z= H\in \mathfrak{a} \subset \mathfrak{s}$. Then,
\begin{equation*}
T_{H} \Ad \left( U\right) \cdot H=\sum_{\alpha \left( H\right) >0}\mathfrak{s}_{\alpha },
\end{equation*}
while
\begin{equation*}
iT_{H}\Ad \left( U\right) \cdot H=\sum_{\alpha \left( H\right) >0}\mathfrak{u}_{\alpha }=\ad \left( Z\right) \left( i\mathfrak{u}\right).
\end{equation*}
By these expressions, it is immediate that $iY$ is tangent to the fiber if $Y$ is tangent to the  orbit  $\Ad \left( U\right) \cdot H$. For $Z=\Ad \left( u\right) \cdot H$, $u\in U$, the same result applying $\Ad \left( u\right)$.
\end{proof}

A vector field in the orbit $\Ad \left( U\right) H$ is an application $ x\mapsto Y\left( x\right) \in \mathfrak{s}$ that assumes values in the tangent space to $x$. By the Proposition above, $iY\left( x\right) \in  \mathfrak{u}$ is tangent to the fiber over the semi-direct orbit. Thus, given a vector field $Y$ in $\Ad \left( U\right) \cdot H$, the vector field $ iY\left( x\right) $ is defined in the semi-direct orbit, such that in the fiber $x+ \ad \left( x\right) \left( \mathfrak{s}\right) $ is a constant field.

\begin{pps}\label{prop: grad} 
Let $Y= \mathrm{grad} f$ be a gradient field in $\Ad \left( U\right) \cdot H$. Thus $iY$ is the Hamiltonian vector field of the function $ \widetilde{f} = f \circ \pi$ with respect to the symplectic form $\Omega_\tau$.
\end{pps}

\begin{proof}
If $W$ is a vertical vector then $d\widetilde{f}\left( W\right) =0$  and $\Omega_\tau  \left( W,iY\left( x\right) \right) =0$, because both $W$ and $iY\left( x\right)$ are in $\mathfrak{u}$.
On the other hand, take a vector of type $\left[ A,x+X\right] =\left[ A,x\right] +\left[ A,X\right]$ with $X\in \ad\left( x\right) \left( i\mathfrak{u}\right) \subset \mathfrak{u}$ (these vectors, together with the vertical space, generate the tangent space as in Proposition \ref{propRestriNaoDegene}). The component $\left[ A,X\right] \in \mathfrak{u}$, so that $d\widetilde{f}\left( \left[ A,X\right] \right) =0$ and $\Omega_\tau \left( \left[ A,X\right] ,iY\left( x\right) \right) =0$. Since the  component $v=\left[ A,x\right]$ is the tangent space to $x$, hence
\begin{equation*}
d\widetilde{f}\left( v\right) =df\left( v\right) =\langle Y\left( x\right)
,v\rangle,
\end{equation*}
because $Y=\mathrm{grad} f$. But,
\begin{equation*}
\Omega_\tau \left( iY\left( x\right) ,v\right) = \mathcal{H}_{\tau }\left( iY\left( x\right) ,v\right) = i\langle Y\left( x\right) ,v\rangle,
\end{equation*}
because in this sequence of equality all terms are purely imaginary. Consequently, for vectors of type $w=\left[ A,x+X \right] =\left[ A,x\right] +\left[ A,X\right]$, holds $d\widetilde{f}\left( w\right) =\Omega_\tau  \left( iY\left( x\right) ,v\right)$, as this equality is also true for vertical vectors, it is shown that $iY\left(x\right)$ is the Hamiltonian vector field of $\widetilde{f}$.
\end{proof} 

\begin{cor}
Let $Y$ be a gradient field on the flag manifold $\mathbb{F}_{H}=\Ad \left( U\right) \cdot H$, and for $t\in \mathbb{R}$ we define the application 
\begin{equation*}
\sigma _{tY}\left( x\right) =x+tiY\left( x\right) .
\end{equation*}
This application is a section of  $U_{\ad} \cdot H$. Then, the image of $\sigma_{tY}$ is a Lagrangian submanifold of $U_{\ad} \cdot H$ with respect to the symplectic form $\Omega_\tau$.
\end{cor}

\begin{proof}
By Proposition \ref{prop: grad}, $iY\left( x\right) $ is a Hamiltonian vector field that is constant in each fiber, which means that if $\sigma_{tY}$ is its flow, then $$\sigma_{tY}\left( x+X\right) =x+X+tiY\left( x\right)$$ to $x+X$ in the fiber over $x$. In particular, the image of $\sigma_{tY}$ on $\mathbb{F}_H$ (0-section)  is a Lagrangian submanifold because the 0-section is Lagrangian, which concludes the demonstration.
\end{proof}

Denote by $L_{tY}$ the image of section $\sigma_{tY}$, which is Lagrangian submanifold of $U_{\ad} \cdot H$. The next step is to find the tangent space to the section $x\mapsto x+iY\left( x\right)$.  If $iY\left( x\right)$ is a section of $U_{\ad } \cdot H \rightarrow \mathbb{F}_H$, then the tangent space $T_{x} \mathbb{F}_H$ is generated by the vectors $\widetilde{A}\left( x\right) =\left[ A,x\right]$ with $A\in  \mathfrak{u}$. Therefore, to determine the space tangent to the section we have to compute the differential of $Y$ in the direction of $\widetilde{A} \left( x\right) =\left[ A,x\right]$. By the formula of the Lie bracket of vector fields we have $dY_{x}\left( \widetilde{A}\left( x\right) \right) =\left[ Y, \widetilde{A}\right] \left( x\right) + d\widetilde{A}_{x} \left( Y\left( x\right) \right)$ and since $\widetilde{A}\left( x\right) = \left[ A,x\right] $ is a linear field it follows that 
\begin{equation}  \label{forDerivadaCampograd}
dY_{x}\left( \widetilde{A}\left( x\right) \right) =\left[ Y,\widetilde{A} \right] \left( x\right) +\left[ A,Y\left( x\right) \right] .
\end{equation}
Multiplying this differential by $i$ and adding the base vector, we get a vector tangent to the image of the section as 
\begin{equation*}
\left[ A,x\right] +i\left[ Y,\widetilde{A}\right] \left( x\right) +i\left[
A,Y\left( x\right) \right] \qquad A\in \mathfrak{u}.
\end{equation*}
These vectors are in fact tangent to the orbit $U_{\ad}\cdot H$ because $\left[ Y,\widetilde{A}\right] \left( x\right) \in T_{x} \mathbb{F}_H$ and therefore $i\left[ Y,  \widetilde{A}\right] \left( x\right)$ is tangent to the fiber over $x$. The sum 
\begin{equation*}
\left[ A,x\right] +i\left[ A,Y\left( x\right) \right] =\left[ A,x+iY\left(
x\right) \right]
\end{equation*}
is tangent to the orbit because $A\in \mathfrak{u}$. This last equality is written as $\ad \left( A\right) \left( \sigma _{Y}\left( x\right) \right)$ where $\sigma_{Y}\left( x\right) =x+iY\left( x\right)$ is the section defined by the field $Y$. Thus the tangent vectors to the section $\sigma_{Y}$ are 
\begin{equation*}
\ad \left( A\right) \left( \sigma _{Y}\left( x\right) \right) +i\left[ Y,%
\widetilde{A}\right] \left( x\right) \qquad A\in \mathfrak{u}.
\end{equation*}
This proves the following characterization of the tangent spaces to the sections.

\begin{pps}
The tangent space to $L_{tY}$ on the section $\sigma _{tY}\left( x\right) =x+itY\left( x\right)$ of $U_{\ad}\cdot H \rightarrow \mathbb{F}_H$ is generated by 
\begin{equation*}
\left[ A,x\right] +ti\left[ Y,\widetilde{A}\right] \left( x\right) +ti\left[
A,Y\left( x\right) \right] =\ad \left( A\right) \left( \sigma_{tY}\left(
x\right) \right) +ti\left[ Y,\widetilde{A}\right] \left( x\right) ,
\end{equation*}
with $A\in \mathfrak{u}$.
\end{pps}

Now, we are interested in using these Lagrangian submanifolds to transport them by a symplectomorphism between $U_{\ad}\cdot H$ and $\Ad(G)\cdot H$, with respect to $\Omega_{\tau}$.
For this, consider the following proposition

\begin{pps}
For $r>0$ we have 
\begin{equation*}
\widetilde{\psi}_r ^{\ast} \left( \Omega_\tau \right) = \Omega_\tau ,
\end{equation*}
that is $\widetilde{\psi}_r$ is symplectomorphism for $r>0$.
\end{pps}

\begin{proof}
If $x \in \Ad(G)\cdot H$ and $\widetilde{x}= \widetilde{\psi}_r (x)$, then
\begin{align*}
\left( \widetilde{\psi}_r ^{\ast} \cdot \Omega_\tau \right)_x (X,Y) &= \left( \Omega_\tau \right)_{\widetilde{x}} \left( (d\widetilde{\psi}_r )_x X, (d\widetilde{\psi}_r )_x Y \right) \\
       &=  \left( \Omega_\tau \right)_{\widetilde{x}} \left( \psi_r  X, \psi_r  Y \right),
\end{align*}
for $X,Y \in \mathfrak{n}_H \simeq T_x \Ad(G)\cdot H$ and their corresponding $\psi_r (X), \psi_r (Y) $ in $\mathfrak{n}_{r, H} \simeq T_{\widetilde{x}} \Ad_r (G)\cdot H$, as seen above $\psi_r (X_\alpha) = X_\alpha ^r$ are the generators for $\alpha \in \langle \Theta_H \rangle$. 
\end{proof}

Similarly, taking $r\rightarrow\infty$, in the manifold $\Ad_\infty (G) \cdot H$ we have 
\begin{equation*}
\widetilde{\psi} ^{\ast} \left( \Omega_\tau \right) = \Omega_\tau ,
\end{equation*}
where $\Omega_\tau $ is a symplectic form of $\Ad_\infty (G) \cdot H$, because as seen above and by Proposition \ref{propRestriNaoDegene} it coincides with $U_{\ad} \cdot H$. Then we conclude that

\begin{teo}\label{simplectomorf} 
The manifolds $K_{\ad }\cdot H$ and $\Ad(G) \cdot H$ are symplectomorphic, with the symplectic form $\Omega_\tau$.
\end{teo}

and consequently

\begin{cor}
The manifolds $\widetilde{\psi} ^{-1} \left( L_{tY} \right)$ are Lagrangian submanifolds of $\Ad(G)\cdot  H$ with the symplectic form $\Omega_\tau$.
\end{cor}

\subsection{Lagrangian submanifolds in the adjoint action}\label{lsadact}

We are interested in finding the isotropic or Lagrangian orbits given by the action of $U$ and its subgroups on $\mathfrak{g}$, or more specifically for $\Ad (G) \cdot H$, where we use some techniques of \cite{bego} and \cite{inftight}. Therefore, the action of $U$  is symplectic in relation to $\Omega_\tau$. We can describe the action in terms of the moment map in $\mathfrak{g}$, and we can specify it in an adjoint orbit. Let $B_{\tau} (X,Y) = - \langle X, \tau Y \rangle_{\mathbb{R}} $ a inner product, where $\langle \cdot , \cdot \rangle_{\mathbb{R}}$ is the Cartan-Killing (real) form of $\mathfrak{g}_{\mathbb{R}}$, that satisfies
\begin{equation*}
    B_{\tau} (X,Y) = 2 \mathrm{Re} \left( \mathcal{H}_{\tau} (X,Y)
\right) \quad \text{and} \quad  B_{\tau} (iX,iY) = B_{\tau} (X,Y),
\end{equation*}
then unless multiplying by $\frac{1}{2}$: 
\begin{equation*}
\Omega_\tau (X,Y) = B_{\tau} (iX,Y) = - \langle iX, \tau Y \rangle_{\mathbb{R}} .
\end{equation*}

For this, we will describe the action in terms of the moment map in $\mathfrak{g}$ and then specifying for the adjoint orbits. So for $A \in \mathfrak{u}$, define the bilinear form: 
\begin{equation*}
\beta_{A} (X,Y) = \Omega_\tau \left( \ad (A) \cdot X, Y \right) = B_{\tau}
\left(i \ad (A) \cdot X, Y \right),
\end{equation*}
such that is symmetric: 
\begin{align*}
\beta_{A} (Y,X) &= B_{\tau} \left(i \ad (A) \cdot Y, X \right) \\
&= B_{\tau} \left(Y, i \ad (A) \cdot X \right) \\
&= B_{\tau} \left( \ad (A) \cdot iX , Y \right) = \beta_{A} (X,Y),
\end{align*}
because $\ad(A)$ is anti-symmetric in relation to $B_\tau$. Then define the quadratic form 
\begin{equation*}
Q(X)= \beta_{A} (X,X) = \Omega_\tau \left( \ad (A)X,X \right).
\end{equation*}

\begin{pps}
If $A \in \mathfrak{u}$ then $\ad(A)$ is a Hamiltonian field with Hamiltonian function $\frac{1}{2} Q(x)$.
\end{pps}

\begin{proof}
Let $\alpha (t)$ be any curve, then 
\begin{align*}
\frac{d}{dt} \left( \frac{1}{2} Q \left( \alpha (t) \right) \right) &= \frac{d}{dt} \left( \frac{1}{2} \beta_{A} \left( \alpha (t), \alpha (t) \right) \right) \\
&= \beta_{A} \left( \alpha ^{\prime }(t), \alpha (t) \right) \\
&= \Omega_\tau \left( \ad (A) \cdot \alpha ^{\prime }(t), \alpha (t) \right).
\end{align*}
therefore a vector field $x \mapsto \ad (A) \cdot x$ is Hamiltonian with function $\frac{1}{2} Q(x)$.
\end{proof}

From this Hamiltonian function we can write the moment map $\mu: \mathfrak{g} \rightarrow \mathfrak{u}$, for $A \in \mathfrak{u}$: 
\begin{equation}
\langle \mu (x), A \rangle_{\mathfrak{u}} = \frac{1}{2} Q(x),
\end{equation}
where $\langle \cdot , \cdot \rangle_{\mathfrak{u}}$ is the Cartan-Killing form of $\mathfrak{u}$. Therefore we have 
\begin{align*}
\langle \mu (x), A \rangle_{\mathfrak{u}} &= \frac{1}{2} \Omega_\tau \left( \ad (A) \cdot x, x \right) \\
&= - \frac{1}{2} \langle i \ad (A) \cdot x, \tau x \rangle_{\mathbb{R}} = \frac{1}{2}
\langle A , [ \tau ix , x] \rangle_{\mathbb{R}}.
\end{align*}
Hence $\mu (x)$ is the orthogonal projection on $\mathfrak{u}$ of $[ \tau ix, x]$, that is 
\begin{equation*}
\mu (x) = \frac{1}{2} \left( [ \tau ix, x] + \tau [\tau ix, x] \right) = [\tau ix, x] \in \mathfrak{u}.
\end{equation*}

\begin{cor}
The moment map $\mu $ for the adjoint action of $U$ in $\mathfrak{g}$ (and thus for the action in each orbit $\Ad(G) \cdot H$) is given to $A\in \mathfrak{u}$ by 
\begin{equation*}
\mu (x)=[\tau ix,x]=-i[\tau x,x]\in \mathfrak{u}\qquad x\in \mathfrak{g}.
\end{equation*}
\end{cor}

From this expression for $\mu$ and \cite[Prop. 4]{inftight}, it follows that the orbit $\Ad(U) \cdot x$ is isotropic for $\Omega_\tau$ if and only if $[ \tau x,x]=0$, since $\mathfrak{u}$ is semisimple. Put another way, $\Ad(U) \cdot x$ is isotropic if and only if $x$ commutes with $\tau x$.

\begin{example}
Let $\mathfrak{g}=\mathfrak{sl}(n,\mathbb{C})$, we have 
\begin{equation*}
\tau x= -x^{*} = \overline{x}^{T} ,
\end{equation*}
therefore the isotropic orbits are the orbits of normal transformations.
\end{example}

One case where the adjoint orbit $\Ad(U) \cdot H$ is isotropic is when $H \in \mathfrak{s}=i \mathfrak{u}$. In this case, $\Ad(U) \cdot H = \mathbb{F}_H$ is a flag manifold of $\mathfrak{g}$. Moreover, we have $\dim \left( \Ad(G) \cdot H \right) = 2 \dim  \mathbb{F}_H$, hence $\mathbb{F}_H$ is Lagrangian submanifold of $\Ad (G) \cdot H$ with respect to $\Omega_\tau$.  Then

\begin{teo}\label{Teo: adjact}
The only isotropic $\Ad(U)$-orbit in $\Ad(G) \cdot H$ is the flag manifold $\mathbb{F}_H$, since it is the only orbit with dimension less or equal to $ \frac{1}{2} \dim \left( \Ad (G) \cdot H \right)$.
\end{teo}

\begin{proof}
It should be proved that if $0 \neq X \in \mathfrak{n}_{H}^{+}$, then the isotropy subgroup $U_{H+X}$ on $H+X$ has a strictly smaller dimension than the dimension of  $U_H$ on $H$, as this shows that  $$\dim \Ad (U) (H+X) > \dim \mathbb{F}_H = \frac{1}{2} \dim \left( \Ad (G) \cdot H \right).$$

For this it is observed that if $$\Ad (u) (H+X) = \Ad (u) \cdot H + \Ad (u) \cdot X = H+X$$ then $\Ad (u) \cdot H =H$ and $\Ad (u) \cdot X =X$. The first equality means that  $U_{H+X} \subset U_H$. Take the torus $T_{H} = \mathrm{cl} \{ e^{itH}: \ t \in \mathbb{R} \}$ which has dimension greater than $0$, then $T_H \subset U_H$ but $\Ad (v) \cdot X \neq X$ for some $v \in T_H$ since $\Ad (T_H)$ has no fixed points in $\mathfrak{n}_{H}^{+}$, because the eigenvalues of $\ad (H)$ in $\mathfrak{n}_{H}^{+}$ are strictly positive.

This shows that $\ad (iH)$ is not in the isotropy algebra $H+X$ and therefore $\dim U_{H+X} < \dim U_H $.
\end{proof}

\section*{Appendix} \label{apendice}

\addcontentsline{toc}{section}{Appendix}


\subsection*{Representations and symplectic geometry}

Let $M\subset W$ be a immersed submanifold of the vector space $W$ (real, that is, $W=\mathbb{R}^{N}$). The cotangent bundle $\pi :T^{\ast
}M\rightarrow M$ is provided with the canonical symplectic form $\omega$. Given a function $f:TM\rightarrow \mathbb{R}$ denote by $X_{f}$ the corresponding Hamiltonian field, such that $df\left( \cdot \right) =\omega \left( X_{f},\cdot \right)$. If $\alpha \in W^{\ast }$, the height function $f_{\alpha }:M\rightarrow \mathbb{R}$ is given by 
\begin{equation*}
f_{\alpha }\left( x\right) =\alpha \left( x\right)
\end{equation*}
and also denote by $f_{\alpha }$ its lifting $f_{\alpha }\circ \pi $ which is constant on the fibers of $\pi$. Denote by $X_{\alpha }$ the Hamiltonian field of this function. Since $f_{\alpha }$ is constant in the fibers, the field $X_{\alpha}$ is vertical and the restriction to the fiber $T_{x}^{\ast}M$   is constant in the direction of the vector $\left( df_{\alpha }\right) _{x}\in T_{x}^{\ast }M$. Furthermore, if $\alpha ,\beta \in W^{\ast }$, the vector fields $X_{\alpha }$ and $X_{\beta }$ commutes.  In terms of the action of Lie groups and algebras, the commutativity $\left[ X_{\alpha },X_{\beta }\right] =0$ means that the application $\alpha \mapsto X_{\alpha }$ is an infinitesimal action of $W^{\ast}$, seen as an abelian Lie algebra. This infinitesimal action can be extended to an action of $W^{\ast }$ (seen as an abelian Lie group because the fields $X_{\alpha }$ are complete).

Now, let $R:L\rightarrow \mathrm{Gl}\left( W\right)$ be a representation of the Lie group $L$ on $W$ and take a $L$-orbit  given by $M=\{R\left( g\right) x:g\in L\}$. The action of $G$ on $M$ lifts to an action in the cotangent bundle $T^{\ast}M$ for linearity. If $\mathfrak{l}$ is the Lie algebra of $L$, then the infinitesimal action of $\mathfrak{l}$ in the orbit  $M$ is given by the fields $y\in M\mapsto R\left( X\right) y$, where $X\in \mathfrak{l}$ and $R\left( X\right) $ also denotes the infinitesimal representation associated to $R$. The infinitesimal action of the lifting in $T^{\ast}M$ is given by $X\in \mathfrak{l}\mapsto H_{X}$, where $H_{X}$ is the Hamiltonian field on $T^{\ast }M$, such that the Hamiltonian function is $F_{X}:T^{\ast }M\rightarrow \mathbb{R}$ given by 
\begin{equation*}
F_{X}\left( \alpha \right) =\alpha \left( R\left( X\right) y\right) \qquad
\alpha \in T_{y}^{\ast }M.
\end{equation*}
The actions of $L$ and $W^{\ast}$ in $T^{\ast }M$ are going to define an action of the semi-direct product $L\times W^{\ast}$, defined by the dual representation $R^{\ast}$. The action of $ L\times W^{\ast}$ on $T^{\ast }M$ is Hamiltonian in the sense that the corresponding infinitesimal action of  $\mathfrak{l}\times W^{\ast}$ is formed by Hamiltonian fields.  When we have a Hamiltonian action we can define its moment application (See \cite[Section 14.4]{smgrlie}). In this case, an application $$m:T^{\ast }M \rightarrow \left(  \mathfrak{l}\times W^{\ast }\right) ^{\ast }=\mathfrak{l}^{\ast }\times W .$$

In the action on $T^{\ast }M$, the field induced by $X\in \mathfrak{l}$ is the Hamiltonian field $H_{X}$ of the function $F_{X}\left( \alpha \right) =\alpha\left( R\left( X\right) y\right)$, while the field induced by $\alpha \in W^{\ast }$ is the Hamiltonian field of the function $f_{\alpha }$. So if  $\gamma \in T_{y}^{\ast } M$, $y\in M\subset W$ then for $ \qquad X\in \mathfrak{l}$ and $ \alpha \in W^{\ast }$
$$m\left( \gamma \right) \left( X\right) =\gamma \left( R\left( X\right)
y\right) \quad \text{and} \quad m\left( \gamma \right) \left( \alpha \right) =\alpha \left( y\right)$$

The first term coincides with the moment $\mu :W\otimes W^{\ast}\rightarrow  \mathfrak{l}^{\ast}$ of the representation $R$, that is, $m\left( \gamma \right) =\mu \left( y\otimes \overline{\gamma }\right)$ such that the restriction of $\overline{ \gamma }\in W^{\ast}$ to the tangent space $T_{y} M$ is equal to $\gamma$. The second term shows that the linear functional $m\left( \gamma \right)$ restricted to $W^{\ast}$ is exactly $y$. Consequently,

\begin{pps}
\label{prop: ap mom} The moment application $m:T^{\ast }M \rightarrow \mathfrak{l}^{\ast}\times W=\mathfrak{l}^{\ast }\oplus W$ is given by 
\begin{equation*}
m\left( \gamma _{y}\right) =\mu \left( y\otimes \overline{\gamma }\right) +y,
\end{equation*}
where $\gamma _{y}\in T_{y}^{\ast }M$ and $\overline{\gamma }\in W^{\ast }$, such that its restriction to $T_{y}M=\{R\left( X\right) y:X\in \mathfrak{l}^{\ast }\}$ is equal to $\gamma $.
\end{pps}

\subsection*{Skew-symmetric bilinear form}

Let $V$ be a vector space (over $\mathbb{R}$ and $\dim V<\infty$) and $\omega $ a skew-symmetric bilinear form in $V$. The radical $R^{\omega}$ of $\omega $ is given by 
\begin{equation*}
R^{\omega }=\{v\in V:\forall w\in V,~\omega \left( v,w\right) =0\}.
\end{equation*}
By definition, $\omega$ is non-degenerate if and only if $R^{\omega }=\{0\}$. In this case $\dim V$ is even and $\omega$ is called a linear symplectic form.

\begin{pps}
\label{propFormaDegenerada} $\omega$ is non-degenerate if and only if there is a maximal isotropic subspace $W$, with $2\dim W=\dim V$.
\end{pps}

\begin{proof}
As it is well known, if $\omega$ is a symplectic form then dimension of the maximal isotropic subspaces (Lagrangian subspaces) is half the dimension of $V$. Furthermore, every isotropic subspace is contained in some Lagrangian subspace. For the reciprocal, take the quotient space $V/R^{\omega}$ and define the form $\overline{\omega}$ in $V/R^{\omega}$ by $\overline{\omega }\left(  \overline{v},\overline{w}\right) =\omega \left( v,w\right)$ which is a skew-symmetric bilinear form in $V/R^{\omega}$. The radical $R^{\overline{\omega }}$ of $\overline{\omega}$ vanishes, because if $\overline{v}\in R^{\overline{\omega }}$ then $\omega \left( v,w\right) =\overline{\omega }\left( \overline{v},\overline{w}\right)=0$, for all $w\in V$. Hence if $\omega$ is not identically null, then $\overline{\omega}$ is a symplectic form.

Now let $W\subset V$ be an isotropic subspace. So, the projection $\overline{W}\subset V/R^{\omega}$ is isotropic subspace for $\overline{\omega}$. If $ W $ is maximal isotropic then $R^{\omega }\subset W$ and as follows from the definition, $\overline{W}$ is  maximal isotropic and therefore $\dim V/R^{\omega }=2\dim \overline{W}$. In this case $\dim W=\dim \overline{W}+\dim R^{\omega}$, then
\begin{eqnarray*}
2\dim W &=&2\dim \overline{W}+2\dim R^{\omega }=\dim V-\dim R^{\omega }+2\dim R^{\omega } \\
&=&\dim V+\dim R^{\omega }.
\end{eqnarray*}
Hence, if $\omega$ is degenerate then $\dim R^{\omega }>0$ and therefore $2\dim W>\dim V$, concluding the demonstration.
\end{proof}


\end{document}